\documentclass[12pt]{amsart}

\usepackage{amssymb,amsbsy}
\usepackage{latexsym,array}
\usepackage{verbatim,color}
\usepackage{fullpage,euscript}
\usepackage[all]{xy}
\usepackage{xcolor}
\usepackage{tikz}
\usetikzlibrary{arrows,shapes,trees}
\usepackage[colorlinks=true,linkcolor=blue,urlcolor=my_color,citecolor=magenta]{hyperref}

\input {cyracc.def}
\font\tencyr=wncyr10 
\def\rus{\tencyr\cyracc}

\numberwithin{equation}{section}

\newtheorem{thm}{Theorem}[section]
\newtheorem{lm}[thm]{Lemma}
\newtheorem{cl}[thm]{Corollary}
\newtheorem{conj}[thm]{Conjecture}
\newtheorem{prop}[thm]{Proposition}

\theoremstyle{remark}
\newtheorem{rmk}{Remark}[section]
\newtheorem{ex}{Example}[section]

\theoremstyle{definition}

\newcommand{\eus}{\EuScript}

\newcommand {\be}{{\mathfrak b}}
\newcommand {\ce}{{\mathfrak c}}

\newcommand {\g}{{\mathfrak g}}

\newcommand {\h}{{\mathfrak h}}
\newcommand {\ka}{{\mathfrak k}}
\newcommand {\el}{{\mathfrak l}}
\newcommand {\me}{{\mathfrak m}}
\newcommand {\n}{{\mathfrak n}}
\newcommand {\fN}{{\mathfrak N}}
\newcommand {\p}{{\mathfrak p}}
\newcommand {\q}{{\mathfrak q}}
\newcommand {\rr}{{\mathfrak r}}
\newcommand {\es}{{\mathfrak s}}
\newcommand {\te}{{\mathfrak t}}
\newcommand {\ut}{{\mathfrak u}}

\newcommand {\fX}{{\mathfrak X}}

\newcommand {\z}{{\mathfrak z}}

\newcommand {\sln}{{\mathfrak{sl}}_n}
\newcommand {\slno}{{\mathfrak{sl}}_{n+1}}

\newcommand {\spn}{{\mathfrak{sp}}_{2n}}

\newcommand {\son}{{\mathfrak {so}}_{n}}

\newcommand{\gt}{\mathfrak}

\newcommand {\gA}{{\eus A}}

\newcommand {\gS}{{\eus S}}
\newcommand {\gZ}{{\eus Z}}

\newcommand {\esi}{\varepsilon}

\newcommand {\ap}{\alpha}

\newcommand {\lb}{\lambda}
\newcommand {\vp}{\varphi}
\newcommand {\tvp}{\tilde\varphi}

\newcommand {\tc}{\tilde c}
\newcommand {\tr}{\tilde r}
\newcommand {\tG}{\tilde G}
\newcommand {\tO}{\tilde \co}
\newcommand {\tH}{\tilde H}
\newcommand {\tme}{\tilde\me}

\newcommand {\hhe}{\hat{\h}}

\newcommand {\co}{{\mathcal O}}

\newcommand {\BQ}{{\mathbb Q}}

\newcommand {\VV}{{\mathsf V}}

\newcommand {\WW}{{\mathsf W}}

\newcommand {\md}{/\!\!/}
\newcommand {\isom}{\stackrel{\sim}{\longrightarrow}}

\newcommand {\codim}{{\mathrm{codim\,}}}

\newcommand {\dif}{{\mathrm{def\,}}}

\newcommand {\ed}{{\mathrm{ed\,}}}

\newcommand {\hd}{{\mathrm{hd\,}}}

\newcommand {\ind}{{\mathrm{ind\,}}}
\newcommand {\Lie}{{\mathsf{Lie\,}}}

\newcommand {\rk}{{\mathrm{rk\,}}}

\newcommand {\spe}{{\mathsf{Spec\,}}}

\newcommand {\trdeg}{{\mathrm{trdeg\,}}}

\newcommand {\tri}{\mathfrak{sl}_2}
\newcommand {\GR}[2]{{\textrm{{\sf\bfseries #1}}}_{#2}}

\newcommand {\ov}{\overline}
\newcommand {\un}{\underline}

\newcommand {\ngv}{{\fN}_G(\VV)}

\newcommand {\bh}{{\boldsymbol{h}}}

\newcommand {\fnme}{\fN_H(\me)}
\newcommand {\fhg}{\fN_H(\g)}
\newcommand {\fng}{\fN_G(\g)}
\newcommand {\fngs}{\fN_G(\g^*)}

\newcommand {\beq}{\begin{equation}}
\newcommand {\eeq}{\end{equation}}

\renewcommand{\le}{\leqslant}
\renewcommand{\ge}{\geqslant}

\newcommand{\odin}{\mathrm{{1}\!\! 1}}

\newcommand {\bbk}{\Bbbk}

\begin{document}
\hfill {\scriptsize May 3, 2024}
\vskip1ex

\title[Coisotropy representation]{Orbits and invariants for coisotropy representations}
\author[D.\,Panyushev]{Dmitri I. Panyushev}
\address{Institute for Information Transmission Problems, Moscow 127051, Russia}
\email{panyush@mccme.ru}
\begin{abstract}
For a subgroup $H$ of a reductive group $G$, let $\me=\h^\perp\subset\g^*$ be the cotangent 
space of $\{H\}\in G/H$. The linear action $(H:\me)$ is the {\it coisotropy representation\/}. It is known 
that the complexity and rank of $G/H$ (denoted $c$ and $r$, respectively) are encoded in properties of 
$(H:\me)$. We complement existing results on $c$, $r$, and $(H:\me)$, especially for quasiaffine 
varieties $G/H$. For instance, if the algebra of invariants $\bbk[\me]^H$ is finitely generated, then 
$\fnme\subset \me\cap\fngs$. Moreover, if $G/H$ is affine, then $\fnme=\me\cap\fngs$ if and only if 
$c=0$. We also prove that the variety $\me\cap\fngs$ is pure, of dimension $\dim\me-r$. Two other 
topics considered are {\sf (i)} a relationship between varieties $G/H$ of complexity at most $1$ and the 
homological dimension of the algebra $\bbk[\me]^H$ and {\sf (ii)} the  Poisson structure of $\bbk[\me]^H$ 
and Poisson-commutative subalgebras $\gA\subset \bbk[\me]^H$ such that $\trdeg\gA$ is maximal.
\end{abstract}
\maketitle

\section*{Introduction}     
\label{sect:intro}

\noindent
In this article, we study invariant-theoretic properties for the coisotropy representation of a homogeneous space of a reductive group $G$.
The ground field $\bbk$ is algebraically closed and $\mathrm{char\,}\bbk=0$. All groups and varieties are 
assumed to be algebraic, and all algebraic groups are affine. If $Q$ is a group and $X$ is a variety, then 
the notation $(Q:X)$ means that $Q$ acts regularly on $X$. We also say that $X$ is a $Q$-{\it variety}. 
Lie algebras of algebraic groups are denoted by the corresponding small gothic letters, e.g., $\q=\Lie Q$. 

Throughout, $G$ is a connected reductive group and $\g=\Lie G$.  
We also consider a Borel subgroup $B\subset G$, the maximal unipotent subgroup $U=(B,B)$, and 
a maximal torus $T\subset B$. This yields a bunch of related objects: roots, weights, simple roots, etc.
For a reductive subgroup $H\subset G$, we denote by $B_H$, $U_H$, and $T_H$ analogous subgroups 
of $H$.

For a subgroup $H\subset G$, let $c=c_G(G/H)$ and $r=r_G(G/H)$ be the {\it complexity\/} and 
{\it rank\/} of the $G$-variety $G/H$, respectively. Then $r\le \rk G$ (see Section~\ref{sect:general} for 
details.) These two integers are important for invariant theory and theory of equivariant embeddings of 
$G/H$. Let $\me=\h^\perp\subset \g^*$ be the cotangent space of $\{H\}\in G/H$. The linear action 
$(H:\me)$ is called the {\it coisotropy representation\/} of $H$ (or $G/H$). It is shown in \cite{p90} that
\\ \indent
\textbullet \  the integers $c$ and $r$ are closely related to properties of $(H:\me)$.
If $G/H$ is quasiaffine, then $\dim\me-\max_{x\in\me}\dim H{\cdot}x=2c+r$, the stabiliser $H^x$ is reductive for generic $x\in\me$,  and $\rk G-\rk H^x=r$.
\\ \indent
\textbullet \ If $c=0$ and $\ce\subset\me$ is a {\it Cartan subspace}, then $\ov{H{\cdot}\ce}=\me$ and 
there is a finite group $W\subset GL(\ce)$ such that $\bbk[\ce]^W\simeq \bbk[\me]^H$. Here the  
morphism $\pi: \me\to \me\md H:=\spe(\bbk[\me]^H)$ is equidimensional, and if
$H$ is connected, then $\bbk[\me]^H$ is a polynomial ring. 

More generally, if $\bbk[\me]^H$ is finitely generated, then $\pi$ is also well-defined and the fibre 
$\pi^{-1}(\pi(0))=:\fnme$ is the {\it nullcone} (in $\me$ with respect to $H$). The nullcone is a fibre
of $\pi$ of maximal dimension and if 
$c>0$, then $\pi$ is not necessarily equidimensional. It is convenient to consider the {\it defect\/} of equidimensionality (= {\it defect\/} of $\fnme$)
\[
   \dif\fnme=\dim\fnme-(\dim\me-\dim\me\md H) .  
\]
If $H$ is reductive, then $\dif\fnme\le c$~\cite[Prop.\,3.6]{p95}.

In Section~\ref{sect:new1}, we present new results  related to the 
nullcones $\fnme$ and $\fngs$, and the generalised Cartan subspace $\ce\subset\me$. 
\\ \indent
{\bf --} \  
For any $x\in \me$, we show that $\dim G{\cdot}x\ge 2\dim H{\cdot}x$, and the equality occurs if and only 
if $\g{\cdot}x\cap\me=\h{\cdot}x$. Another general property is that
$\dim \bigl(\me\cap\fngs\bigr)=\dim\me -r$ and all irreducible components of $\me\cap\fngs$ have this dimension (Theorem~\ref{thm:peresechenie}).
\\ \indent
{\bf --} \ {\em If\/ $G/H$ is quasiaffine}, then one can define a generalised Cartan subspace 
$\ce\subset\me$ (see Section~\ref{sect:general}) and we prove that
$\codim_\me \ov{H{\cdot}\ce}=c$.
\\ \indent
{\bf --} \ {\em If\/ $G/H$ is quasiaffine and $\bbk[\me]^H$ is finitely generated}, then 
$\fnme\subset \me\cap\fngs$. Moreover, if $c=0$, then $\fnme=\me\cap\fngs$ 
(Theorem~\ref{thm:null-c}).  
\\ \indent
{\bf --} \ {\em If\/ $G/H$ is affine}, then $\fnme=\me\cap\fngs$ if and only if $G/H$ is spherical
(Prop.~\ref{prop:granitsy}). We also prove that if $\dif\fnme=c$, then $\h$ contains a regular semisimple
element of $\g$ (Theorem~\ref{thm:s-reg}). In particular, this applies to the affine homogeneous spaces
with $c=0$.

In Section~\ref{sect:new2}, affine homogeneous spaces of the form $\tO=(G\times H)/\Delta_H$ are studied. Here $H\subset G$ is reductive and $\Delta_H$ is the diagonal in 
$H\times H\subset G\times H=\tG$. Then $\tO\simeq G$ and the isotropy representation of 
$\Delta_H\simeq H$ is identified with the $H$-module $\g$. In this 
case, $\tO$ has a group structure, $r_{\tG}(\tO)=\rk\tG$ is maximal, and $\dim\fhg$ can be 
computed via a result of R.\,Richardson~\cite{r89}.
Here we present some complements to results of Section~\ref{sect:new1}.

In Section~\ref{sect:defect}, we consider coisotropy representations $(H:\me)$ with $\dif\fnme\le 1$. If
$c=0$, then $\pi$ is equidimensional and $\me\md H\simeq\mathbb A^n$. This can be regarded as an
illustration to the Popov conjecture~\cite{po76}. In~\cite{p95}, we stated a related conjecture that if $H$ is connected reductive and $c=1$, then $\me\md H$ is either an affine space or a hypersurface. We verify this in two cases:

(a) \ for the homogeneous spaces $G/H$ with simple $G$;

(b) \ for the homogeneous spaces $\tO=(G\times H)/\Delta_H$ with $c_{\tG}(\tO)=1$, where $G$ is simple.
\\
In both cases, classifications of such pairs $(G,H)$ are known (see~\cite{p92} for (a) and \cite{ap02} 
for (b)), and we perform a case-by-case verification.

In Section~\ref{sect:poisson}, $H$ is reductive and the natural Poisson bracket $\{\ ,\,\}$ on the affine 
variety $\me\md H$ is considered. Let $\gZ$ be the Poisson centre of $(\bbk[\me]^H, \{\ ,\,\})$. Then there 
is the natural morphism $\boldsymbol{f}:\me\to \spe\gZ$. Using  results of F.\,Knop~\cite{kn90},
we prove that $\boldsymbol{f}$ is equidimensional and $\boldsymbol{f}^{-1}(0)=\me\cap\fngs$. 

It is shown that if a subalgebra $\gA\subset \bbk[\me]^H$ is Poisson-commutative, then $\trdeg \gA\le 
c+r$. We conjecture that there is always such a subalgebra with $\trdeg \gA= c+r$. Some partial results
towards this conjecture are described.

Our basic reference for Invariant Theory is~\cite{vp89}.

\noindent
{\bf Data availability and conflict of interest statement.}

This article has no associated data. There is no conflict of interest.

\section{Generalities on group actions and coisotropy representations}
\label{sect:general}

\noindent
Let $\bbk[X]$ denote the algebra of regular functions on a variety $X$. If $X$ is irreducible, then 
$\bbk(X)$ is the field of rational functions on $X$. If $X$ is acted upon by $Q$, then $\bbk[X]^Q$ and 
$\bbk(X)^Q$ are the subalgebra and subfield of invariant functions, respectively. The identity component 
of $Q$ is denoted by $Q^o$. For $x\in X$, let $Q^x$ denote the {\it stabiliser} of $x$ in $Q$. Then 
$\q^x=\Lie Q^x$. A stabiliser $Q^x$ is said to be {\it generic}, if there is a dense open subset 
$\Omega\subset X$ such that $Q^y$ is $Q$-conjugate to $Q^x$ for all $y\in\Omega$. We say that a 
property ({\bf P}) holds for {\it almost all points of\/} $X$, if there is a dense open subset $X_0\subset X$ 
such that ({\bf P}) holds for all $x\in X_0$.

\subsection{Complexity and rank}
\label{subs:c-and-r}
Let $X$ be an irreducible $G$-variety. Then 
\\ \indent {\bf --} \ the {\it complexity} of $X$ is $c_G(X) =\dim X-\max_{x\in X}\dim B{\cdot}x$,   
\\ \indent  {\bf --} \
the {\it rank} of $X$ is $r_G(X) =\max_{x\in X}\dim B{\cdot}x- \max_{x\in X}\dim U{\cdot}x$.

By the Rosenlicht theorem (see e.g. \cite[\S\,2.3]{vp89}), we also have
\[
    \text{$c_G(X)=\trdeg\, k(X)^{B}$ and $c_G(X)+r_G(X)=\trdeg\, k(X)^{U}$}.
\]
An alternate approach to the rank uses the weights of $B$-semi-invariants in $\bbk(X)$. For quasiaffine 
varieties, this boils down to the following. Write $\fX_+=\fX_+(G)$ for the set of dominant weights of $G$ 
with respect to $(B,T)$. Let $\VV_\lb$ denote a simple $G$-module with highest weight $\lb\in\fX_+$. Let 
$\bbk[X]=\bigoplus_{\lb\in\fX_+}\Bbbk[X]_{(\lb)}$ be the sum of $G$-isotypic components, where
$\Bbbk[X]_{(0)}=\bbk[X]^G$. Then 
\[
   \Gamma_X=\{\lb\in \fX_+\mid \bbk[X]_{(\lb)}\ne 0\}
\]
is the {\it rank monoid} of $X$ and $r_G(X)=\dim_\BQ(\BQ\Gamma_X)$. Clearly $r\le \rk\g$. If 
$c_G(X)=0$, then $X$ is said to be a {\it spherical\/} $G$-variety. If $\bbk[X]^G=\bbk$, then 
$\dim \bbk[X]_{(\lb)}<\infty$ for all $\lb$ and $m_\lb(X):=\dim \bbk[X]_{(\lb)}/\dim \VV_\lb$ is the {\it multiplicity} 
of $\VV_\lb$ in  $\bbk[X]$ (=\,the multiplicity of $\lb$ in $\Gamma_X$). 

\subsection{The coisotropy representation}  
\label{subs:coiso}
Let $H$ be an algebraic subgroup of $G$ with $\Lie H=\h$. Then $\g/\h\simeq\mathsf{T}_{\{H\}}(G/H)$ is 
an $H$-module and the linear action $(H:\g/\h)$ is the {\it isotropy representation} of $H$. Set 
$\me=\h^\perp=\{\xi\in\g^*\mid \xi\vert_\h=0\}$. Then $\me\simeq (\g/\h)^*$ as $H$-module and the linear 
action $(H:\me)$ is the {\it coisotropy representation} of $H$. If necessary, we identify $\g$ 
and $\g^*$ using a non-degenerate $G$-invariant bilinear form $\Psi$ on $\g$ and regard  $\me$ 
as subspace of $\g$. 

Recall that, for a reductive $G$, $G/H$ is affine if and only if $H$ is reductive. Another equivalent 
condition is that the form $\Psi$ is non-degenerate on $\h$. In this case, $\me\simeq \g/\h$ and 
$\g\simeq \h\oplus\me$ as $H$-module. Then 
the linear action $(H:\me)$ will be referred to as the {\it isotropy representation} of $H$.

We always assume that the $G$-action on $G/H$ has a finite kernel.
This is tantamount to saying that $H$ contains no infinite normal subgroups of $G$. This 
condition is always satisfied, if $G$ is simple. 
If $G/H$ is spherical, then $H$ is said to be a {\it spherical} subgroup of $G$. For simplicity, we write $c$ 
and $r$ for the complexity and rank of the homogeneous space $G/H$. 

\begin{thm}[{~\cite{p90}, \cite[Ch.\,2]{p99}}]    \label{thm:coiso-props}
If\/ $G/H$ is quasiaffine, then 
\begin{enumerate}
\item \ There is a generic stabiliser for $(H:\me)$, say $S$, which is reductive;
\item \ $\dim G+\dim S-2\dim H=\dim\me-\max_{x\in\me}\dim H{\cdot}x=2c+r$;
\item \ $\rk G-\rk S=r$.
\end{enumerate}
\end{thm}
The theory developed in~\cite{p90} (and presented with more details in~\cite{p99}) contains much more 
results. We mention those that will be needed later. Let us assume that $B$ and $T$ are fixed. Then the 
choice of $H$ (up to conjugacy in $G$) and $S$ (up to conjugacy in $H$) is at our disposal. It was proved 
that $H$ and $S$ can be chosen such that
\begin{itemize}
\item[($\eus P_1$)] \ $Z_G(t)' \subset S\subset Z_G(t)$ for some $t\in\te=\Lie T$. Hence
$T\subset N_G(S)$ and $S/S^o$ is abelian; 
\item[($\eus P_2$)] \ $\be\cap \es$ is a Borel subalgebra of $\es$ and $B\cap S$ is a generic stabiliser for $(B:G/H)$;
\item[($\eus P_3$)] \ $\ut\cap \es$ is the nilradical of $\be\cap\es$ and $U\cap S$ is a generic stabiliser for $(U:G/H)$;
\item[($\eus P_4$)] \ $\te\cap\es$ is a Cartan subalgebra of $\es$; 
\item[($\eus P_5$)] \ $B{\cdot}S=P$ is a parabolic subgroup and $P\cap H=S$.
\end{itemize}

\vskip.1ex\noindent
Whenever the algebra $\bbk[\me]^H$ is finitely generated, we consider the following objects:

\textbullet \ \ the {\it categorical quotient} \ $\me\md H:=\spe(\bbk[\me]^H)$; 
\\ \indent
\textbullet \ \ the {\it quotient morphism\/} $\pi=\pi_{H,\me}:\me\to \me\md H$ induced by the inclusion
$\bbk[\me]^H\hookrightarrow \bbk[\me]$;
\\ \indent
\textbullet \ \ the {\it nullcone}\/ $\fN_H(\me):=\pi^{-1}(\pi(0))\subset\me$. 

\begin{ex}    \label{ex:kr71}
Let $\sigma\in {\sf Aut}(\g)$ be an involution and $\g=\g_0\oplus\g_1$ the sum of $\pm 1$-eigenspaces 
of $\sigma$. If $G_0$ is the connected subgroup of $G$ with $\Lie G_0=\g_0$, then $G/G_0$ is affine 
and $c_G(G/G_0)=0$. We say that $G/G_0$ is a {\it symmetric variety} and $G_0$ is a {\it symmetric
subgroup\/} of $G$. The isotropy representation $(G_0:\g_1)$ has thoroughly been studied by 
Kostant--Rallis~\cite{kr71}. For instance, they proved that $\bbk[\g_1]^{G_0}$ is a polynomial algebra, 
$\dim\g_1\md G_0=r_G(G/G_0)$, and $\pi_{G_0,\g_1}:\g_1\to\g_1\md G_0$ is {\it equidimensional}, i.e., 
all fibres have the same dimension.
\end{ex}

For any quasiaffine $G/H$, we introduced in~\cite{p90} a certain subspace $\ce\subset\me$, which is 
useful in the study of the linear action $(H:\me)$. Let us recall the general 
construction of $\ce$. The definitions of the complexity and rank of $G/H$ imply that 
\[
    \dim G/H-\max_{x\in G/H}\dim B{\cdot}x=c \quad \& \quad \dim G/H-\max_{x\in G/H}\dim U{\cdot}x=c+r .
\]
Without loss of generality, we may assume that $x=\{H\}$ is generic in both senses and 
properties ($\eus P_1$)-($\eus P_5$) are satisfied. 
Then $\codim_\g(\be+\h)=c$, $\codim_\g(\ut+\h)=c+r$, and we set 
\[
  \text{  $\ce=(\be+\h)^\perp=\ut^\perp\cap\me$ \ \ \& \ \ $\tilde\ce=(\ut+\h)^\perp=\be^\perp\cap\me$.}
\]
Then $\dim\ce=c+r$ and  $\dim\tilde\ce=c$.
It follows that $\dim\ce\le \dim\me-\max_{x\in\me}\dim H{\cdot}x=2c+r$, and the equality occurs if 
and only if $c=0$. Upon identification of $\g$ and $\g^*$, we have
$\ce=\be\cap\me$ and $\tilde\ce=\ut\cap\me$. 

Consider the projection $p_\te:\be=\te\oplus\ut\to \te$ and set 
$\te_1= p_\te(\ce)\subset\te$. Then $\dim \te_1=r$. 
If $c=0$, then $\tilde\ce=\{0\}$ and $p_\te$ maps $\ce$ isomorphically to $\te_1$. In this case,  
$\ce$ contains no nilpotent elements of $\g$. Moreover,  the following holds.

\begin{thm}[{\cite[Section\,3.2]{p90}}]    
\label{thm:Cartan-sph}
For $c=0$, the subspace $\ce\subset\me$ has the following properties:
\begin{enumerate}
\item \ the $H$-saturation of $\ce$ is dense in $\me$, i.e., $\ov{H{\cdot}\ce}=\me$;
\item \ almost all elements of $\ce$ have the same stabiliser in $H$, which is just $S$;
\item \ there is a finite group $W\subset GL(\ce)$ such that the restriction homomorphism $\bbk[\me]\to\bbk[\ce]$ induces an isomorphism $\bbk[\me]^H\isom \bbk[\ce]^W$.
\end{enumerate}
\end{thm}

\noindent
It follows from (1) and (2) that, for almost all $x\in\ce$, the stabiliser $H^x$ is generic, while (3)
implies that $\bbk[\me]^H$ is finitely generated and $\pi:\me\to\me\md H$ is {\it equidimensional}.
The common dimension of fibres equals $\dim\me-\dim\me\md H=\dim\me-r$.
Furthermore, if $H$ is connected, then $W$ is a reflection group and $\bbk[\me]^H$ is a polynomial ring, 
i.e., $\me\md H\simeq \mathbb A^r$ is an affine space, see~\cite[Cor.\,5]{p90}.

If $c=0$, then $\ce$ resembles a Cartan subspace for the isotropy representation a symmetric 
variety $G/G_0$. For this reason, a subspace $\ce$ satisfying properties of 
Theorem~\ref{thm:Cartan-sph} was christened in \cite{p90} a {\it Cartan subspace} (of $\me$). 

Thus, Theorem~\ref{thm:Cartan-sph} shows that many good properties of the symmetric variety $G/G_0$ 
and $(G_0:\g_1)$ are retained for quasiaffine spherical homogeneous spaces.

For an arbitrary quasiaffine $G/H$, we shall say that $\ce=\ut^\perp\cap\me$, as above, 
is a {\it generalised Cartan subspace\/} of $\me$. If $G/H$ is not spherical, then $\ce$ does not satisfy properties (1) and (3) of Theorem~\ref{thm:Cartan-sph} (cf. also Theorem~\ref{thm:codim=c} below).

{\it\bfseries Remarks.} 1)  In~\cite{p90}, the generalised Cartan subspace is denoted by $\z$.
Here we follow the notation of \cite[Chapter\,2]{p99}.
\\ \indent
2) Above results on the complexity, rank, and coisotropy representations are also obtained by 
Knop~\cite{kn90} via different methods. Our approach in~\cite{p90,p99} is based on the study of `doubled 
actions' (which is not discussed here), while Knop considers the cotangent bundles and moment map. 

The nullcone $\fnme$ is a fibre of $\pi$ of maximal dimension~\cite[\S\,5.2]{vp89}. The action $(H:\me)$ 
is said to be {\it equidimensional}, if the quotient morphism $\pi:\me\to \me\md H$ is equidimensional. The 
{\it defect\/} of equidimensionality of $\pi$ (=\,of $\fN_H(\me)$) is introduced in~\cite[Section\,3]{p95} as 
the difference between $\dim\fnme$ and dimension of generic fibres of $\pi$, i.e.,
\\[.6ex]
\centerline{ $\dif\fN_H(\me)=\dim\fN_H(\me)-(\dim\me-\dim\me\md H)$.}

\noindent
Therefore, $\pi$ is equidimensional if and only if $\dif\fnme=0$.

If $H$ is reductive, then $\fN_H(\me)=\{x\in\me\mid \ov{H{\cdot}x}\ni 0\}$ and the representation 
$(H:\me)$ is orthogonal. The latter implies that the action $(H:\me)$ is stable~\cite{lu72} and therefore
$\dim\me\md H=2c+r$. Then $\dim\fnme\ge \dim\me-(2c+r)$. On the other hand, there is an
upper bound on dimension of the nullcone for the self-dual representations of reductive 
groups~\cite[Prop.\,2.10]{g78b}. In our case, this shows that
\beq         \label{eq:gerry}
      \dim\fnme\le \dim U_H+\frac{1}{2}(\dim\me-\dim\me^{T_H})=\frac{1}{2}(\dim\g-\dim\g^{T_H}) .
\eeq

\begin{thm}[{\cite[Proposition\,3.6]{p95}}]    \label{thm:def-N-c}
If $H$ is reductive, then\/ $\dif\fN_H(\me)\le c$.
\end{thm}

\begin{proof}[Outline of the proof]
$1^o$.  If $r=\rk\g$, then $S$ is finite and $\dim G/H=\dim\me=\dim B+c$. 
Hence $\dim\fN_H(\me)\ge\dim\me-\dim\me\md H=\dim H=\dim U-c$. On the other hand,
\beq     \label{eq:esche-odno}   \textstyle
         \dim\fN_H(\me) \le \frac{1}{2}(\dim\g-\dim\g^{T_H}){\le} \dim U .
\eeq
$2^o$. If $r <\rk\g$, then $\Lie S\ne 0$ and one can use the Luna--Richardson theorem~\cite[Theorem\,4.2]{lr79}. This provides a (rather technical) reduction to part $1^o$. 
\end{proof}

\begin{rmk}   \label{rem:dim-H}
The relations $\dim H=\dim U-c$ and $\dim \me=\dim B+c$ hold only if  $S$ is finite. 
In general, one has $\dim H=\dim U+\dim B_S-c$ and $\dim\me=\dim B-\dim B_S+c$.
\end{rmk}

\subsection{Homogeneous spaces of complexity at most 1}
The Luna--Vust theory of equivariant embeddings of homogeneous spaces (1983) implies that a reasonably complete theory can be developed for homogeneous spaces of complexity $\le 1$. For a modern account of that theory and related topics, we refer to \cite{tim}.

As is already mentioned, if $\h$ is a fixed point subalgebra for an involution of $\g$, then $c_G(G/H)=0$. All connected spherical reductive subgroups $H$ of {\bf simple} algebraic groups $G$
have been found by M.\,Kr\"amer~\cite{kr79}. Then M.\,Brion and I.\,Mikityuk (independently) found all connected spherical reductive subgroups of the {\bf semisimple} algebraic groups, see e.g. tables in~\cite[Ch.\,I, \S\,3.6]{v01}.

The study of quasiaffine homogeneous spaces of complexity~1 was initiated in~\cite{p92}, where a 
classification of the pairs $(G,H)$ such that $G$ simple, $H$ is connected reductive, and $c_G(G/H)=1$ 
is also obtained. (See also~\cite[Chapter\,3]{p99}.)

\section{New results for coisotropy representations}
\label{sect:new1}

\noindent
For the symmetric varieties (see Example~\ref{ex:kr71}),  one has $\dim G{\cdot}x=2\dim G_0{\cdot}x$ 
for {\bf any} $x\in\g_1$~\cite[Prop.\,5]{kr71}. For quasiaffine spherical $G/H$, this equality holds 
generically, i.e., there is a dense open subset $\Omega\subset\me$ such that 
$\dim G{\cdot}x=2\dim H{\cdot}x$ for all $x\in\Omega$~\cite[Theorem\,5]{p90}. 
Then $H{\cdot}x$ is a Lagrangian subvariety of the symplectic variety $G{\cdot}x\subset\g^*$ for all 
$x\in\Omega$. The following observation is a slight extension of~\cite[Proposition\,1]{p90}.

\begin{lm}   \label{lm:twice-dim}
 {\sf (i)} For {\bfseries any} algebraic subgroup $H\subset G$ and $x\in\me=\h^\perp$, one has
\\[.5ex]
\centerline{   $\dim G{\cdot}x=\dim H{\cdot}x+\dim([\g,x]\cap\me)\ge 2\dim H{\cdot}x$.}
\vskip.3ex

{\sf (ii)}   $\dim G{\cdot}x=2\dim H{\cdot}x \Longleftrightarrow  [\g,x]\cap\me=[\h,x]$.
\end{lm}
\begin{proof}  We have
$([\g,x]\cap\me)^\perp=([\g,x])^\perp+\me^\perp=  \g^x+\h$. Hence 
\\[.6ex]
 \centerline{$\dim([\g,x]\cap\me)=\dim\g-\dim\g^x-\dim\h+\dim(\g^x\cap\h)=\dim G{\cdot}x-\dim H{\cdot}x$.}
 
 It is also clear that \ $[\g,x]\cap\me\supset [\h,x]$.
\end{proof}

\noindent 
In the setting of symmetric varieties (Example~\ref{ex:kr71}), the relation $[\g,x]\cap\g_1=[\g_0,x]$ for 
all $x\in\g_1$ readily follows from the presence of involution $\sigma$.

Recall that $\ce=\me\cap\ut^\perp=\me\cap\be$ is a generalised Cartan subspace, 
$\tilde\ce=\ce\cap\ut$, and $\te_1=p_\te(\ce)\subset\te$. The following is a generalisation of Theorem~\ref{thm:Cartan-sph}(1).

\begin{thm}        \label{thm:codim=c}
Let $G/H$ be a quasiaffine homogeneous space and $\ce\subset\me$ a generalised Cartan subspace.
Then \ $\codim_\me \ov{H{\cdot}\ce}=c_G(G/H)=c$. 
\end{thm}
\begin{proof}  
1. We assume that properties ($\eus P_1$)-($\eus P_5$) are satisfied for $S,B$, and $H$.
Then $\el=\es\oplus\te_1$ is a Levi subalgebra of $\p=\Lie P$. Moreover, by~\cite[Lemma\,3]{p90}, one 
has $\g^y=\el$ for almost all $y\in\te_1$. Let $\n$ denote the nilradical of $\p$ and $\n_-$ the opposite nilradical, i.e., 
$\g=\p\oplus\n_-$. Then $\ce=\me\cap\be=\me\cap(\n\oplus\te_1)$ and $\tilde\ce=\me\cap\ut=\me\cap\n$.

2. By part~1, we have $\ce^\perp=\me^\perp+(\n\oplus\te_1)^\perp=\h+(\es\oplus\n)=\h+\n$. For 
almost all $y\in\ce$,  $\g^y$ is a Levi subalgebra of $\p$ (see Step~1 in \cite[Theorem\,2.2.6]{p99}). 
Therefore $ ([\g,y]\cap\ce)^\perp=\g^y+\ce^\perp=\h+(\g^y+\n)=\h+\p$. Hence
\beq      \label{eq:gy-cap-ce}
     [\g,y]\cap\ce=(\h+\p)^\perp=\me\cap\n=\tilde\ce \ \text{ and } \ \dim([\g,y]\cap\ce)=c .
\eeq
\indent
3. Let us prove that $[\h,y]\cap\ce=\{0\}$ for almost all $y\in\ce$. (For $c=0$, this follows 
from~\eqref{eq:gy-cap-ce}, and this was already used in~\cite{p90,p99} for proving 
Theorem~\ref{thm:Cartan-sph}(1).) By part~2, $[\g,y]\cap\ce=\tilde\ce\subset\n$. Therefore, it suffices to 
prove that $[\h,y]\cap\n=\{0\}$. Actually, we shall show that $[\h,y]\cap\p=\{0\}$.

Since $\h\cap\p=\es$, we can write $\h=\es\oplus\hhe$, where $\hhe\cap\p=\{0\}$. Then every nonzero
element of $\hhe$ has a nonzero component in $\n_-$ w.r.t. the sum $\g=\p\oplus\n_-$. Write $y=y'+y''$, 
where $y'\in\te_1$ and $y''\in\n$, i.e.,  $p_\te(y)=y'$. Let us say that $y'\in\te_1$ is generic, if 
$\g^{y'}=\el$. Take a nonzero $x=x_\es+\hat x\in \h$, where $x_\es\in\es$ and $\hat x\in\hhe$. Write 
$\hat x=x_\p+x_-$, where $x_\p\in\p$ and $x_-\in\n_-$. If $\hat x\ne 0$, then $x_-\ne 0$ as well. We have
\[
      [x,y]=[\hat x, y]=[x_\p,y]+[x_-,y] 
\]
and here $[x_\p,y]\in \p$. Since $y'$ is a generic element of $\te_1$, there is $n\in N=\exp(\n)$
such that $n{\cdot}y=y'$. Then $n{\cdot}[x_-,y]=[n{\cdot}x_-,y']$. Again, since $y'$ is generic, the last 
bracket has a nonzero component in $\n_-$. Therefore, the same holds for
$[x_-,y]=n^{-1}{\cdot}[n{\cdot}x_-,y']$.

Thus, we proved that 
$[\h,y]\cap\p=\{0\}$ for almost all $y\in\ce$ and thereby $[\h,y]\cap\ce=\{0\}$.

4. Since $[\h,y]\cap\ce=\{0\}$ for almost all $y\in\ce$, the intersection $H{\cdot}y\cap\ce$ is finite.
Hence
\begin{multline*}
   \dim \ov{H{\cdot}\ce}=\dim H{\cdot}y+\dim\ce=\dim H-\dim S+(c+r)\\ 
    =\dim\me-(2c+r)+(c+r)=\dim\me-c .     \qedhere
\end{multline*}
\end{proof}

\begin{thm}   \label{thm:peresechenie} 
For {\bfseries any} homogeneous space $G/H$, we have 

(1) \ \ $\dim \bigl(\me\cap\fngs\bigr)=\dim\me-r$ and
all irreducible components of\/ $\me\cap\fngs$ have this dimension.

(2) \ \ The intersection $\me\cap\fngs$ is proper if and only if $r=\rk\g$.
\end{thm}
\begin{proof}
(1) To use some results of F.\,Knop~\cite{kn90}, we need more notation. Let 
$T^*_{\co}\simeq G\times_H\me$ be the cotangent bundle of $\co=G/H$ and 
$\widetilde\Phi: T^*_\co\to \g^*$  the associated moment map. Let $\widetilde M_\co=\ov{G{\cdot}\me}$ 
denote the closure of the image of $\widetilde\Phi$ in $\g^*$. Finally, $M_\co$ is the spectrum of the integral closure of $\bbk[\widetilde M_\co]$
in $\bbk[T^*_\co]$. This yields the commutative diagram of morphisms
\beq   \label{diagr-knop}
\xymatrix{
T^*_\co\ar^(.45)\psi[dr]\ar^\Phi[r]\ar@/_3pc/@{->}_{\tilde\psi}[rrd]\ar@/^2pc/@{->}^{\widetilde\Phi}[rr] 
& M_\co\ar^{\tilde\tau}[r]\ar^{\pi_M}[d] & 
\widetilde M_\co\ar@{^{(}->}^i[r]\ar^{\pi_{\tilde M}}[d] &  \g^*\ar^{\pi_{\g^*}}[d] \\
 & M_\co\md G\ar^{\tau}[r] & \widetilde M_\co\md G\ar@{^{(}->}[r] &  \g^*\md G 
}
\eeq
where the vertical arrows are quotient morphisms and $\tilde\Phi=\tilde\tau\circ\Phi$~see~\cite[Sect.\,6]{kn90}. By construction,
$\widetilde\tau$ is finite and onto. Then so is $\tau$. It is proved in~\cite{kn90} that 
\begin{itemize}
\item \ $M_\co\md G$ is an affine space, see~Satz\,6.6(b);
\item \ $\psi$ is equidimensional and onto, see~Satz\,6.6(c);
\item \ $\dim M_\co\md G=r$, see~Satz\,7.1.
\end{itemize} 
Then $\tilde\psi=\tau\circ\psi$ is equidimensional and onto, too. Hence, for $\bar 0=\pi_{\tilde M}(0)$, we 
obtain $\dim\tilde\psi^{-1}(\bar 0)=\dim T^*_\co-r=2\dim\me-r$. Note that 
$\pi_{\tilde M}^{-1}(\bar 0)=\ov{G{\cdot}\me}\cap\fngs$. Hence
\[
   \tilde\psi^{-1}(\bar 0)=\widetilde\Phi^{-1}(\ov{G{\cdot}\me}\cap\fngs)=G\times_H(\me\cap\fngs) .
\]
Therefore, $\dim(\me\cap\fngs)=2\dim\me-r-\dim G/H=\dim\me-r$ and all irreducible components have this dimension, as required.

(2) By definition, the intersection of $\me$ and $\fngs$ in $\g^*$ is proper if and only if 
\[\dim 
\bigl(\me\cap\fngs\bigr)=\dim\me+\dim\fngs-\dim\g=\dim\me-\rk\g .  \qedhere
\]
\end{proof}

If $H$ is reductive, then $\fN_H(\me)=\{m\in\me\mid \ov{H{\cdot}m}\ni 0\}$. Therefore,
$\fN_H(\me)\subset \me\cap\fngs$. As we prove below, this inclusion actually holds in a more general
situation.

\begin{thm}                    \label{thm:null-c}
Suppose that $G/H$ be quasi-affine  and\/ $\bbk[\me]^H$ is finitely generated. Then
\begin{itemize}
\item[\sf (i)] \   $\fN_H(\me)\subset \me\cap\fngs$;
\item[\sf (ii)] \  if $c=0$, then $\fN_H(\me)= \me\cap\fngs$.
\end{itemize}
\end{thm}
\begin{proof} 
{\sf (i)} \ Consider a version of the commutative diagram~\eqref{diagr-knop} :
\beq     \label{diagr-knop2}
\xymatrix{
\fnme\ar@{^{(}->}[r] & \me\ar@{^{(}->}^j[r]\ar^{\pi_\me}[d] &
G{\times}_H\me \ar^{\tilde\Phi}[r]\ar^(.45){\tilde\psi}[dr]\ar[d] 
& \ov{G{\cdot}\me}\ar^{\pi_{\tilde M}}[d] \ar@{^{(}->}[r] & \g^* \\
& \me\md H \ar^(.4){\sim}[r] & (G{\times}_H\me)\md G\ar^f[r] & \ov{G{\cdot}\me}\md G &
}
\eeq 
Here $\ov{G{\cdot}\me}=\widetilde M_\co$ and $j$ is the embedding of $\me$ as fibre of $\{H\}\in G/H$.
As above, all vertical arrows are quotient morphisms. Since $\bbk[\me]^H$ is finitely generated, we get
two new objects in the lower row. Using the path through $\pi_\me$, we see that $\fnme$ maps into
$\bar 0=\pi_{\tilde M}(0)\in \ov{G{\cdot}\me}\md G$. On the other hand, using the path through $j$ and
$\tilde\psi$,  we see that $j(\fnme)\subset \tilde\psi^{-1}(\bar 0)=G\times_H\bigl(\me\cap\fngs\bigr)$, i.e.,
$\fnme\subset \me\cap\fngs$.

{\sf (ii)} \ Here we use a fragment of the previous diagram:
\begin{center}
$\xymatrix{
 \me\ar@{^{(}->}^{\tilde\Phi\circ j}[r]\ar^{\pi_\me}[d] & \ov{G{\cdot}\me}\ar^{\pi_{\tilde M}}[d] \\
 \me\md H \ar^f[r] & \ov{G{\cdot}\me}\md G 
}$
\end{center} 
If $c=0$, then $\dim \me\md H=\dim\ov{G{\cdot}\me}\md G=r$. Since $\tilde\psi$ is equidimensional, the same is true for $f$. The affine varieties $\me\md H$ and $\ov{G{\cdot}\me}\md G$ are conical, hence
$f$ is finite and $\pi_\me(0)=f^{-1}(\bar 0)$. Therefore,
$\fnme=\ov{G{\cdot}\me}\cap \fngs\cap\me=\fngs\cap\me$.
\end{proof}

\begin{rmk}  \label{rem:another-proof}
(a) There is an alternate proof of Theorem~\ref{thm:null-c}{\sf (ii)} that exploits properties of the Cartan 
subspace $\ce=\be\cap\me$. 
In the spherical case, the projection $p_\te:\be\to \te$ maps $\ce$ isomorphically onto $\te_1$ and 
one proves that $\ov{G{\cdot}\te_1}=\ov{G{\cdot}\ce}=\ov{G{\cdot}\me}$.  Then the finiteness of the morphism
$f:\me\md H\to \ov{G{\cdot}\me}\md G$ is obtained without using the equidimensional map $\tilde\psi$.

(b) For the symmetric varieties, the equality $\fN_{G_0}(\g_1)=\g_1\cap\fng$ is proved in~\cite{kr71}.
\end{rmk}

\noindent
For the affine homogeneous spaces, one can strengthen Theorem~~\ref{thm:null-c} as follows.

\begin{prop}        \label{prop:granitsy}
If\/ $G/H$ is affine, then 
\begin{itemize}
\item[\sf (i)] \ $\dim\me-2c-r \le \dim \fnme\le  \dim\me-c-r$;
\item[\sf (ii)] \ $\fnme=\me\cap\fngs$ \ if and only if \ $G/H$ is spherical.
\end{itemize}
\end{prop}
\begin{proof} 
{\sf (i)} \ The first inequality means that $\dim\fnme\ge \dim\me-\dim\me\md H$. 
By Theorem~\ref{thm:def-N-c}, if $G/H$ is affine, then
\[
     \dif\fnme =\dim\fnme-(\dim\me -(2c+r))\le c .
\] 
Hence the second inequality.

{\sf (ii)} \ It follows from part {\sf (i)} and Theorem~\ref{thm:peresechenie} that 
{$\dim\fnme<\dim \bigl(\me\cap\fngs\bigr)$} unless $c=0$.
\end{proof}

\begin{rmk}   \label{rem:2-svoistva}
The equality in the first place in Proposition~\ref{prop:granitsy}{\sf (i)} is equivalent to that $\pi$ is
equidimensional, while the equality in the second place means that $\dim\fnme=c$, i.e., it is maximal 
possible. Hence, for $c=0$,  both properties hold (as we already know). For $c=1$ exactly one property 
takes place, i.e., either $\pi$ is equidimensional, or $\dif\fnme=1$.
\end{rmk}

Let us say that a reductive subgroup $H\subset G$ is $s$-{\it regular}, if $\h$ contains a regular semisimple element of $\g$.

\begin{prop}       \label{prop:2.6}
Suppose that $G/H$ is affine and $S^o$ is a torus. If $\dif\fnme=c$, then $H$  is $s$-regular. 
\end{prop}
\begin{proof}
If $S^o$ is a torus, then $\dim H=\dim U+\dim S-c$ (see Remark~\ref{rem:dim-H}) and 
$\dim\me-\dim\me\md H=\dim H-\dim S$. Then 
\[
   \dif\fnme=\dim\fnme-\dim H+\dim S=\dim\fnme-\dim U+c .
\]
If $\dif\fnme=c$, then $\dim\fnme=\dim U$. By Eq.~\eqref{eq:esche-odno}, this means that  
$\dim \g^{T_H}=\rk\g$, i.e., $T_H$ contains a regular semisimple element of $G$. 
\end{proof}

In particular, Proposition~\ref{prop:2.6} asserts that if $c=0$ and $S^o$ is a torus, then $H$ is 
$s$-regular. However, it is easily seen that if a symmetric subgroup $G_0$ does not contain infinite
normal subgroups of $G$ (e.g. if $G$ is simple), then it is $s$-regular regardless of the structure of $S$. 
This suggests that the condition on $S$ in Proposition~\ref{prop:2.6} is superfluous. 

\begin{thm}    \label{thm:s-reg}
If \ $G/H$ is affine and $\dif\fnme=c$, then $H$ is $s$-regular. 
\end{thm}
\begin{proof}
By Proposition~\ref{prop:2.6}, it suffices to handle the case in which $r<\rk\g$, i.e., $\rk S>0$. 

Let $\mathcal N_A(B)$ (resp. $\mathcal Z_A(B)$) denote the {\it normaliser} (resp. {\it centraliser}) of the 
group $B$ in $A$. Since $S$ is contained between a Levi subgroup of $G$ and its commutant, there is a 
connected reductive group $K\subset G$ such that $\mathcal N_{G}(S)^o=S^o{\cdot} K$ and 
$S^o\cap K$ is finite. Then $\rk K=\rk G-\rk S=r$. Since 
$\mathcal N_{G}(S)^o=\mathcal Z_G(S)^o{\cdot} S^o$, one also has 
$\mathcal  Z_G(S)^o=K{\cdot}Z(S^o)$, where $Z(S^o)$ is the centre of $S^o$. For Lie algebras, this
means that $\z_\g(\es)=\ka\oplus \z(\es)$.
 
As  $S\subset H$, we have $N_{H}(S)^{o}=S^o{\cdot}(K\cap H)^{o}$ and 
$K\cap H$ is reductive. The linear action $(K\cap H:\me^S)$ is the coisotropy representation of the affine 
homogeneous space $K/K\cap H$, and it follows from the construction that it's generic stabiliser is finite. 

By the Luna--Richardson theorem~\cite[Theorem\,4.2]{lr79}, the restriction homomorphism 
$\bbk[\me]\to\bbk[\me^S]$, $f\mapsto f\vert_{\me^S}$, induces an isomorphism 
$\bbk[\me]^H\isom\bbk[\me^S]^{N_H(S)}$. Hence 
$\bbk[\me^S]^{N_H(S)^0}=\bbk[\me^S]^{K\cap H}$ is a finite $\bbk[\me]^H$-module and 
$\fnme\cap \me^S=\fN_{K\cap H}(\me^S)$. It is also known that

{\bf --} \ $c_K(K/K\cap H)= c_G(G/H)= c$, see~\cite[1.9]{MAnn} and  
\\ \indent
{\bf --} \ $\dif\fnme\le \dif \fN_{K\cap H}(\me^S)$~\cite[Lemma\,3.4]{p95}.

\noindent It then follows from Theorem~\ref{thm:def-N-c} that $\dif \fN_{K\cap H}(\me^S)=c$. Therefore,
Proposition~\ref{prop:2.6} applies to $K/K\cap H$ and we conclude that $\ka\cap\h$ contains regular semisimple elements of $\ka$.  Let $\te_S$ be a Cartan subalgebra of $\es$. (Then $\te_S\subset\h$.)
Let $\te_K$ be a Cartan subalgebra of $\ka$ such that $\te_K\cap\h$ contains a regular semisimple element of $\ka$.
Since $K$ and $S^o$ commute and their intersection is finite, $\te_S\oplus\te_K=:\tilde\te$ is
a Cartan subalgebra of $\g$.
Let $x\in\te_S$ be a sufficiently general semisimple element of $\es$. Then
$\z_\g(x)=\z_\g(\te_S)\supset \z_\g(\es)=\ka\oplus \z(\es)$. Consequently, 
$\z_\g(\te_S)=\te_S+ \tilde \ka$ for a reductive Lie algebra $\tilde\ka \supset [\ka,\ka]$. However,
since $\es$ contains a commutant of a Levi subalgebra of $\g$, it is not hard to prove that  
$\tilde\ka = [\ka,\ka]$. In other words, 
\[
   [\z_\g(\te_S),\z_\g(\te_S)]=[\z_\g(\es),\z_\g(\es)]=[\ka,\ka] .
\]
Therefore, if $x\in\te_S$ is sufficiently general and $y\in\te_K\cap\h$ is regular in $\ka$, then
$x+y\in \h$ is a regular semisimple element of $\g$. 
\end{proof}

\begin{cl}
If $G/H$ is an affine spherical homogeneous space, then $H$ is $s$-regular.
\end{cl}
\begin{rmk}   \label{rem:can-happen}
(1) If $G/H$ is quasiaffine but not affine, then the condition $c=0$ 
does not guarantee that $H$ is $s$-regular. For instance, take $H=U$.

(2) It can happen that $c=1$, $S$ is finite, and $H$ is $s$-regular, but  $\dif\fnme=0<1$. For instance, 
take $G=SO_{2n+1}$ and $H=SL_n$.
\end{rmk}

\section{On a class of affine homogeneous spaces}
\label{sect:new2}

\noindent
For a reductive subgroup $H\subset G$, we set $\tG=G\times H$ and 
$\tH=\Delta_H=H\times H\subset \tG$. Then 
$\tO=\tG/\tH$ is an affine homogeneous space of the reductive group $\tG$, which is isomorphic to $G$. 
The transitive $\tG$-action on $G$ is given by the formula $(g,h)\circ s=gs h^{-1}$, where $(g,h)\in\tG$ 
and $s\in G$. Then the space of the coisotropy representation of $\tH$ is
\[
    \tme=\tilde\h^\perp=\{(\xi,\eta)\in\g^*\times\h^* \mid \xi\vert_\h=-\eta\}.
\]
We shall identify the $\tH$-module $\tme$ with the $H$-module $\g^*$ via the projection to the first 
factor in $\tilde\g^*=\g^*\times\h^*$. If we use the isomorphisms $\g^*\simeq\g$ and $\h^*\simeq\h$, and
the sum $\g=\h\oplus\me=\h\oplus\h^\perp$, then 
\beq       \label{eq:tilde-coiso}
    \tme=(\me\times\{0\}) \oplus \{(x,-x)\mid x\in\h\} \subset \g\times \h .
\eeq

The multiplicities in the algebra $\bbk[\tG/\tH]$ are closely related to the multiplicities in the branching rule 
$G\downarrow H$. Let $\tilde\lb=(\lb,\mu)$ be a dominant weight of $\tilde G$, where $\lb\in\fX_+(G)$ 
and $\mu\in\fX_+(H)$. Write $\tilde\VV_{\tilde\lb}=\VV_\lb\otimes \WW_\mu$ for a simple 
$\tilde G$-module. Let $m_{\tilde\lb}(\tO)$ denote the multiplicity of $\tilde\VV_{\tilde\lb}$ in $\bbk[\tO]$, 
i.e., the multiplicity of $\tilde\lb$ in the rank monoid $\tilde\Gamma=\Gamma_{\tO}$. Then 
$m_{\tilde\lb}(\tO)$ equals the multiplicity of $\WW_\mu$ in $\VV_\lb^*\vert_H$, see~\cite{ap02}. It was 
also shown therein that if $c_{\tilde G}(\tilde O)\le 1$, then one can explicitly describe all the multiplicities 
and the branching rule.

Set $\tilde c=c_{\tilde G}(\tilde O)$, $\tilde r=r_{\tilde G}(\tilde O)$. By~\cite[Section\,3]{ap02} one has
$\tc=c_G(G/B_H)$ and   
\beq
         \tr=\rk \tilde G=\rk G+\rk H.        \label{eq:tr}
\eeq
If $H$ does not contain infinite normal subgroups of $G$ (in particular, $H$ is a proper subgroup of $G$), then there is a more practical formula
\beq     \label{eq:tc}
     \tc=\dim U-\dim B_H .
\eeq
Note that if $H=G$, then $\tO=G\times G/\Delta_G$ is a spherical homogeneous space of $G\times G$, i.e., $\tc=0$. That is, the constraint on $H$ is necessary for~\eqref{eq:tc} to be valid. Whenever we consider the homogeneous space $\tO=\tG/\tH$, it is also assumed that $H$ does not contain infinite normal subgroups of $G$ and thereby \eqref{eq:tc} holds.

We have here the quotient morphism $\tilde\pi:\tilde\me\simeq\g\to \g\md H\simeq \tme\md\tH$ 
and the nullcone $\tilde\pi^{-1}(\tilde\pi(0))=\fN_{\tilde H}(\tilde \me)\simeq\fN_H(\g)$.
It follows from \eqref{eq:tr} and \eqref{eq:tc} that 
\beq          \label{eq:dim-g/H}
       \dim\g\md H=2\tilde c+\tilde r=\dim\g-\dim\h=\dim\me .
\eeq
\begin{prop}   \label{lm:null-equi}  
All irreducible components of the nullcone $\fN_H(\g)$ have the same dimension, which is equal 
to \  $\dim U_H+\frac{1}{2}(\dim\g-\dim\g^{T_H})$.
\end{prop}
\begin{proof}
Since the $H$-module $\g$ contains $\h$ as a summand, all roots of $\h$ occur as $H$-weights in $\g$.
Moreover, $\g$ is a self-dual $H$-module. Therefore, $\g$ satisfies the conditions (C.1) and (C.3)
considered by Richardson in~\cite{r89}, and his Theorem~7.3 applies here.
\end{proof}
\noindent
The point of this result is that the upper bound on dimension of the nullcone given in~\cite{g78b}, 
cf.~Eq.~\eqref{eq:gerry}, provides now the exact value. Properties of $\tG/\tH$ are better than those of 
arbitrary affine homogeneous spaces $G/H$, because $\tr_{\tG}(\tG/\tH)=\rk\tG$ and $\dim\fN_H(\g)$ is 
known. Preceding formulae also show that  
\[
     \dif\fN_H(\g)=\frac{1}{2}(\dim\g-\dim\g^{T_H})-\dim B_H\le \dim U-\dim B_H=\tc ,
\]
which illustrates to the easy part of Theorem~\ref{thm:def-N-c}.
Since $\g^{T_H}$ contains a Cartan subalgebra of $\g$ and $\dim \fN_H(\g)\ge\dim\g-\dim\g\md H=\dim H$, Proposition~\ref{lm:null-equi} implies that 
\beq          \label{eq:2-ineq}
   \textstyle \dim B_H\le \frac{1}{2}(\dim\g-\dim\g^{T_H})\le  \dim U.
\eeq
Here the equality in the first place is equivalent to that $\dim\fN_H(\g)=\dim H$, i.e., $\tilde\pi$ is 
equidimensional. Whereas the equality in the second place is equivalent to that $\g^{T_H}$ is a Cartan 
subalgebra of $\g$, i.e., $H$ is $s$-{regular}. It is easily seen that
{$H$ is $s$-regular in $G$ if and only if $\tH$ is $s$-regular in $\tG$.}

Comparing equations~\eqref{eq:tc} and \eqref{eq:2-ineq} shows that 
\begin{itemize}
\item if \ $\tc=0$, then $H$ is $s$-regular and $\tilde\pi$ is equidimensional;  
\item for $\tc=1$, exactly one of these two properties is satisfied.
\end{itemize}
However, for homogeneous spaces $\tG/\tH$ and the isotropy representation $(H:\g)$,
there is a more precise assertion for any $\tc>0$.

\begin{thm}    
\label{thm:no-EQ}
Suppose that $\g$ is simple, $\tilde c>0$, and\/ $\h\ne 0$. Then the first inequality in~\eqref{eq:2-ineq} 
is always strict, i.e., $\tilde\pi$ cannot be equidimensional. In particular, if \ $\tc=1$, 
then $H$ is $s$-regular and $\dif\fN_H(\g)=1$.
\end{thm}
\begin{proof} Without loss of generality, we may assume that $H$ is connected. Let $x\in\h$ be a  
semisimple element such that $H^x=T_H$. The orbit $H{\cdot}x\subset\g$ is closed and, since 
$\g=\h\oplus\me$, the slice representation at $x$ equals $(T_H:\te_H\oplus\me)$. Here $\te_H$ is a
trivial $T_H$-module. The property of being equidimensional is inheritable, see~\cite[\S\,8.2]{vp89}.
Therefore, if $(H:\g)$ is equidimensional, then so are $(T_H:\te_H\oplus\me)$ and $(T_H:\me)$.

Assume that  $\tilde\pi$ is equidimensional, i.e., $\dim B_H=\frac{1}{2}(\dim\g-\dim\g^{T_H})$. Then we 
have $\dim U=\frac{1}{2}(\dim\g-\dim\g^{T_H})+\tc$ and hence 
\beq      \label{dim-T-fixed}
         \dim\g^{T_H}=\rk \g+2\tc .
\eeq
Take a 1-parameter subgroup $\lb:\bbk^\times\to T_H$ such that $\me^{T_H}=\me^{\lb(\bbk^\times)}$. 
Then
\[
       \me=\me^+\oplus \me^{T_H}\oplus \me^- ,
\]
where $\me^+$ (resp. $\me^-$) is the sum of weight spaces $\me_\nu$ such that $(\lb,\nu)>0$
(resp. $(\lb,\nu)<0$). Since $\me$ is a self-dual $H$-module,  
the $T_H$-weights in $\me^+$ and $\me^-$ are opposite to each other. As 
$\g^{T_H}=\te_H\oplus \me^{T_H}$, it follows from~\eqref{dim-T-fixed} that
$\dim \me^{T_H}=\rk\g-\rk\h+2\tc$. Then using~\eqref{eq:tr} and \eqref{eq:dim-g/H}, we obtain 
$\dim\me=\rk\g+\rk\h+2\tc$ and $\dim\me^+=\dim\me^-=\rk\h$.

The equidimensional representations of tori are described by Wehlau~\cite{w92}. For the self-dual 
representations, his description implies that the nonzero weights in $\me^+$ are linearly independent. 
Therefore, the nonzero $T_H$-weights in $\me^+$ (and in $\me^-$) are of multiplicity~1. As the same is true for the $T_H$-weights in $\h$, we obtain the following conditions:
\begin{itemize}
\item[($\lozenge_1$)] \ the multiplicity of any nonzero $T_H$-weight in $\g$ is \ $\le 2$;
\item[($\lozenge_2$)] \  the number of weights with multiplicity $2$ is at most $2\rk\h=
\dim(\me^+\oplus\me^-)$.
\end{itemize}

\noindent  
Let us prove that ($\lozenge_1$) and ($\lozenge_2$) cannot be satisfied if  
$\h\ne\{0\}$. By~\eqref{dim-T-fixed}, $\el=\g^{T_H}$ is not abelian. Without loss of generality, we 
may assume that $\el$ is a standard Levi subalgebra w.r.t. $T\subset B$, i.e., $\el$ is determined by the 
set of simple roots $\ap$ such that $\ap\vert_{\te_H}=0$.

(a) Assume that $[\el,\el]$ has a simple factor of rank $\ge 2$. Then there is a chain of 
simple roots $\ap_1,\ap_2,\beta$ in the Dynkin diagram of $\g$ such that $\ap_i\vert_{\te_H}=0$ ($i=1,2$)
and $\beta\vert_{\te_H}\ne 0$.
Then $\beta, \beta+\ap_2, \beta+\ap_2+\ap_1$ have the same (nonzero) restriction to $\te_H$, which contradicts ($\lozenge_1$).

(b)  Assume that $[\el,\el]\simeq k\GR{A}{1}$ and $k\ge 2$.  Take simple roots $\ap_1,\ap_2$ in $[\el,\el]$
such that the simple roots between them, say $\beta_1,\dots,\beta_r$, do not belong to $[\el,\el]$. If 
$\beta=\sum_{i=1}^r\beta_i$, then the roots $\beta,\beta+\ap_1, \beta+\ap_2$ yield a 
$T_H$-weight of multiplicity $\ge 3$, which again contradicts ($\lozenge_1$).

(c)  Assume that $[\el,\el]\simeq \GR{A}{1}$. Then $\tc=1$, generic elements of $\te_H$ are 
subregular in $\g$, and there is a unique root $\ap\in\Pi$ such that $\ap\vert_{\te_H}=0$. 
Each pair of roots of the form $\{\mu, \mu+\ap\}$ gives rise to a $T_H$-weight in $\g$ of multiplicity $2$. 

$\bullet$ \ 
If there are roots of different length and $\ap$ is short, then one can find a triple of roots
$\mu, \mu+\ap,\mu+2\ap$, which again provides a $T_H$-weight of multiplicity $\ge 3$.

$\bullet$ \ 
For $\ap$ long, the number of pairs of {\bf positive} roots $\{\mu, \mu+\ap\}$
equals $\bh^*{-}2$, where $\bh^*$ is the {\it dual Coxeter number\/} of $\g$, see~\cite[Section\,1]{p06}. 
Then the total number of such pairs equals $2(\bh^*-2)$ and~($\lozenge_2$) means that  
$2\rk\h\ge 2(\bh^*-2)$.  Since $\bh^*-1\ge\rk\g$, one must have
\[
       \bh^*-2\le \rk\h<\rk\g\le \bh^*-1 .
\]
Hence $\rk\g=\bh^*-1$ and $\rk\h=\rk\g-1$. 
The equality $\rk\g=\bh^*-1$ holds only for $\GR{A}{n}$ and $\GR{C}{n}$, and we look more carefully at these two series.  Since $\tc=1$ and $\rk\h=\rk\g-1$, we obtain using~\eqref{eq:dim-g/H} that 
\[
   \dim \h=\dim\g-\rk\g-\rk\h-2\tc=\dim\g-2\rk\g-1 .
\]
\begin{description}
\item[$\g=\slno$]  Here $\dim\h=n^2-1$ and hence $\h=\sln$ is the only possibility. But the subgroup
$SL_n\subset SL_{n+1}$ is $s$-regular, if $n\ge 2$, i.e., if $\h\ne 0$.
\item[$\g=\spn$]  Here $\dim\h=2n^2-n-1 > \dim \mathfrak{sp}_{2n-2}$, and this case is also impossible. 
\end{description}
Thus, the assumption that $\tilde\pi$ is equidimensional 
leads to a contradiction.
\end{proof}

\begin{rmk}        \label{rem:specific}
This result is specific for homogeneous spaces of the form $(G\times H)/\Delta_H$. For arbitrary 
affine homogeneous spaces $G/H$, it can happen that $c=1$, but $\pi:\me\to \me\md H$ is 
equidimensional and $H$ is not $s$-regular. For instance, take $(G,H)=(Sp_{2n}, Sp_{2n-2})$ or
$(SO_{2n+1}, SO_{2n-1})$ with $n\ge 2$.
\end{rmk}

\subsection{More on the nullcone for $\tilde\pi$}
For $\tG=G\times H$ and $\tG/\tH\simeq G$, we have 
\\[.4ex]
{\phantom{a}\hfil $\fN_{\tH}(\tilde\me)\subset \fN_{\tG}(\tilde\g)\cap\tilde\me$
\ and \ $\fN_{\tG}(\tilde\g)=\fN_G(\g)\times\fN_H(\h)\subset \g\times\h$. \hfil} 

Using~\eqref{eq:tilde-coiso}, one readily verifies that under the isomorphism $\tme\simeq\g$ the variety  
$\fN_{\tG}(\tilde\g)\cap\tilde\me$ is identified with $\fN_G(\g)\cap\bigl(\fN_H(\h)\times\me\bigr) \subset \g$. 
Since $\tr=\rk\tilde\g$, translating Theorem~\ref{thm:peresechenie}, Theorem~\ref{thm:null-c}, and 
Proposition~\ref{prop:granitsy} into this setting, we obtain

\begin{thm}  \label{thm:tilde-null}
If \ $\tc=0$, then $\fN_H(\g)=\fN_G(\g)\cap(\fN_H(\h)\times\me)$.
For arbitrary $\tc\ge 0 $, we have 
\begin{itemize}
\item[\textbullet]  \ $\dim\fN_H(\g)\le \dim\h+\tc=\dim \tilde U=\dim U+\dim U_H$;  
\item[\textbullet] \  $\dim(\fN_G(\g)\cap(\fN_H(\h)\times\me))= \dim\h+2\tc=\dim\g-\rk\g-\rk\h$;
\item[\textbullet] \ the intersection\/ $\fN_G(\g)\cap(\fN_H(\h)\times\me)$ is proper.
\end{itemize}
Moreover, if\/ $\dim\fN_H(\g)= \dim\h+\tc$, i.e., $\dif\fN_H(\g)=\tc$, then $H$ is $s$-regular in $G$.
\end{thm}

\subsection{Homogeneous spaces $\tG/\tH$ of complexity $\le 1$} 
\leavevmode

\noindent
The pairs $(G,H)$ such that $c_{\tG}(\tG/\tH)=0$ can be characterised by a number of equivalent properties. An extensive list of such properties is given and discussed in~\cite[Section\,3.2]{kot-kruks}.

In particular, $(G, H)$ is a {\it strong Gelfand pair}, which means that any simple $G$-module $\VV_\lb$
is a multiplicity free $H$-module. A classification of strong Gelfand pairs (in the category of compact Lie groups) is obtained by Manfred Kr\"amer in \cite{kr76}. 

If $G$ is simple, then the (very short) list of strong Gelfand pairs consists of two series:
\\[.6ex]
\centerline{$(\sln,\mathfrak{gl}_{n-1})$, $n\ge 2$, and $(\son,\mathfrak{so}_{n-1})$, $n\ge 5$.}

\noindent
See also comments in~\cite[Section\,4]{ap02} and other details in~\cite[Section\,3.2]{kot-kruks}.

\begin{rmk}   \label{rem:Andrus}
For a symmetric variety $G/G_0$ with simple $G$, it is proved in~\cite{andr} that 
\\[.6ex]
\centerline{$\tilde\pi:\g\to \g\md G_0$ is equidimensional $\Longleftrightarrow$ $(\g,\g_0)$ is either $(\sln,\mathfrak{gl}_{n-1})$ or $(\son,\mathfrak{so}_{n-1})$.}

\noindent
Since these two series gives rise to the only spherical homogeneous spaces $\tG/\tH$, our 
Theorem~\ref{thm:no-EQ} generalises that result of~\cite{andr}.
\end{rmk}

The list of pairs $(G,H)$ such that $G$ is simple, $H$ is connected, and $\tc=1$ is obtained 
in~\cite[Section\,4]{ap02}. For the reader convenience, we recall it in Table~\ref{tc=1}.

\begin{table}[ht]
\caption{The pairs $(G,H)$ with simple $G$ and $\tc=1$}   \label{tc=1}
\begin{center}
\begin{tabular}{c|| >{$}c<{$}>{$}c<{$}>{$}c<{$}>{$}c<{$}>{$}c<{$}>{$}c<{$}>{$}c<{$}>{$}c<{$}|}
 {\rus N0} & 1&2&3&4&5&6&7&8 \\ \hline
 $G$ & SL_{n+1} & Sp_6             & Spin_7      & \GR{G}{2} & SL_3 & SL_3 & Sp_4 & SL_4 \\
 $H$ & SL_n & Sp_4{\times} SL_2 & \GR{G}{2} & SL_3 & SO_3 & T  & SL_2{\cdot}\bbk^\times & (SL_2)^2{\cdot}\bbk^\times \\  \hline
\end{tabular}
\end{center}
\end{table}

\noindent
For {\rus N0}\,6, $T=\mathsf T_2$ is a maximal torus of $SL_3$, and
{\rus N0}\,7 represents actually two different pairs. Here $H$ is a Levi subgroup of $Sp_4$, and there are two non-conjugate Levi subgroups corresponding to either the {\bf long} or the {\bf short} simple root of $Sp_4$. In Section~\ref{sect:defect}, these two cases will be referred to as 7($l$) and 7($s$), respectively.

\section{The defect of the null-cone and invariants}
\label{sect:defect}

Let $G\to GL(\VV)$ be a linear representation of a connected reductive group $G$. In~\cite{po76}, 
V.\,Popov conjectured that if $G$ is semisimple and $\pi:\VV\to \VV\md G$ is equidimensional, then 
$\VV\md G$ is an affine space, i.e, $\bbk[\VV]^G$ is a polynomial algebra.
Afterwards, this conjecture was extended to arbitrary connected reductive groups.
Using our terminology, the conjecture can be stated as follows:

\emph{ If $G$ is a connected reductive group and \ $\dif\ngv=0$, then $\VV\md G$ is an 
affine space.}

\noindent
There had been a good few classification work related to this conjecture. To the best of my knowledge,
it is verified in the following cases:
\begin{itemize}
\item \ $G$ is simple, $\VV$ is irreducible (V.L.\,Popov, 1976~\cite{po76});
\item \ $G$ is simple, $\VV$ is reducible  (O.M.\,Adamovich, 1980);
\item \ $G$ is semisimple, $\VV$ is irreducible (P.\,Littelmann, 1989);
\item \ $G$ is a torus (E.B.~Vinberg (oral lecture at MSU) 1983; D.\,Wehlau, 1992~\cite{w92}); 
\item \ $G$ is a product of two simple factors, with some exceptions  (D.\,Wehlau, 1993~\cite{w93}).
\end{itemize}
\noindent
An interesting approach to an {\sl a priori\/} proof of the Popov conjecture is presented in~\cite{v82}.
More information on this conjecture and other references can be found in \cite[\S\,8.7]{vp89}.

Some time ago, I stated a similar conjecture on non-equidimensional representations. Let $\ed\VV\md G$
denote the {\it embedding dimension\/} of $\VV\md G$, i.e., the minimal number of generators of 
$\bbk[\VV]^G$. Then $\hd \VV\md G:=\ed\VV\md G- \dim\VV\md G$ is the {\it homological dimension\/} 
of $\VV\md G$, see~\cite{po83}.

\begin{conj}[{\cite[Conj.\,3.5]{p95}}]      \label{conj:def=1}
Suppose that\/ $G$ is connected, $\VV$ is a self-dual $G$-module, and $\dif\ngv=1$. Then $\VV\md G$ 
is either an affine space or a hypersurface, i.e., $\hd \VV\md G\le 1$.
\end{conj}

The assumption on self-duality is essential here, see Example in~\cite[p.\,94]{p95}. This conjecture
is proved for tori~\cite[Prop.\,3.10]{p95} and $G=SL_2$~\cite[Example\,3.12(1)]{p95}.

The isotropy representation of a reductive subgroup $H\subset G$ is orthogonal and the complexity of 
$G/H$ provides an upper bound on $\dif\fnme$, see Theorem~\ref{thm:def-N-c}. Therefore, 
Conjecture~\ref{conj:def=1} can be specialised to the following

\begin{conj}      \label{conj:coiso}
If $H$ is connected reductive and $c_G(G/H)=1$, then\/ $\hd \me\md H\le 1$. 
\end{conj}

As a support to Conjecture~\ref{conj:coiso}, we prove below two theorems. Before stating these 
theorems, we describe our general approach. If $H$ is simple and $H\subset GL(\VV)$, then we use 
various classification results on the structure of $\VV\md H$:
\begin{itemize}
\item when $\VV\md H\simeq \mathbb A^N$ \cite{ag79, g78a};
\item when $\VV\md H$ is a complete intersection, especially a hypersurface~\cite{shm, shm01};
\item when $H=SL_2$ and $\hd (\VV\md SL_2)\le 3$ \cite[Theorem\,4]{po83}.
\end{itemize}
If $H$ is not simple, then we work with consecutive quotients, using factors of $H$. Suppose that
$H=H_1\times H_2$ is a product of reductive groups. Then one has the quotient morphisms
\beq    \label{eq:chain}
    \me\stackrel{\pi_1}{\longrightarrow} \me\md H_1\stackrel{\pi_2}{\longrightarrow} 
    (\me\md H_1)\md H_2=\me\md H .
\eeq
In all cases below, we can choose $H_1$ such that $\hd (\me\md H_1)\le 2$, hence $\me\md H_1$ is a 
complete intersection. 
(In most cases, we actually obtain $\hd (\me\md H_1)\le 1$.) This yields an 
embedding $\me\md H_1\hookrightarrow \VV_2$ into an $H_2$-module $\VV_2$ such that 
$\codim_{\VV_2} (\me\md H_1)\le 2$.  
Then using the equations of $\me\md H_1$ in $\VV_2$, we describe $\me\md H$ as a subvariety of 
$\VV_2\md H_2$. This allows us to handle the second step in \eqref{eq:chain} and prove that
$\hd (\me\md H)\le 1$.

To describe $\me$ as $H$-module, we need some notation. The fundamental weights of $H$ are 
denoted by $\{\vp_i\}$ and $\esi$ stands for the basic character of one-dimensional torus $\bbk^\times=\mathsf T_1$. 
The fundamental weights for the second (resp. third) simple factor of $H$ are marked with prime (resp. double prime). 
The unique fundamental weight of $SL_2$ is denoted by $\vp$. Write $\odin$ for the trivial one-dimensional representation.

As in~\cite{g78a,g78b,shm01}, the simple $H$-module $\WW_\lb$ is identified with its highest weight $\lb$, using 
the multiplicative notation for $\lb$ in terms of the fundamental weights. For instance, we write 
$\vp_j\vp_k+3\vp_i^2$ in place of $\WW_{\vp_j+\vp_k}+3\WW_{2\vp_i}$. Finally, $\lb^*$ is a dual 
$H$-module to $\lb$.

\begin{thm}     \label{thm:confirm-c1}
If $G$ is simple and $c_{G}(G/H)=1$, then 
\begin{itemize}
\item either \ $\me\md H$ is an affine space and $\dif\fnme=0$; 
\item or \ $\me\md H$ is a hypersurface and $\dif\fnme=1$.
\end{itemize}
(Hence an \emph{a priori} conceivable case, where $\me\md H\simeq \mathbb A^n$ and $\dif\fnme=1$, does not occur.)
\end{thm}
\begin{proof}
The list of such pairs $(G,H)$ consists of 17 items, and we refer to their numbering in~\cite[Table\,1]{p92}
(see also Table\,1 in~\cite{p99}).  The output is that $\bbk]\me]^H$ is a polynomial algebra for {\rus N0}\,1,\,4--9,\,13,\,16,\.17. For the other 
cases, $\bbk]\me]^H$ is a hypersurface.

Let us provide some details to our computations. If $H$ is simple, then the pairs with a polynomial 
algebra $\bbk]\me]^H$ can be picked from the list of "coregular representations" of $H$ obtained by Schwarz~\cite{g78a} and Adamovich--Golovina~\cite{ag79}. This applies to {\rus N0}\,4--8,\,13,\,16--17. Moreover, for all these cases, one also has $\dif\fnme=0$, see~\cite{g78b}.

For {\rus N0}\,1, we have $(G,H)=(SL_{2n}, SL_n\times SL_n)$ and 
$\me=\vp_1\vp'_1+(\vp_1\vp'_1)^*+\odin$. Here one can use the fact that 
$\hat H=(SL_n)^2{\cdot} \mathsf T_1$ is a symmetric subgroup of $G$, with isotropy representation
$\hat\me=\vp_1\vp'_1\esi+(\vp_1\vp'_1\esi)^*$,  and hence 
$\bbk[\hat\me]^{\hat H}$ is a polynomial algebra.

For {\rus N0}\,9, we have $(G,H)=(\GR{C}{n}, \GR{C}{n-2}\times \GR{A}{1}\times \GR{A}{1})$, $n\ge 3$, 
and  $\me=\vp_1\vp'+\vp_1\vp''+\vp'\vp''$.
\\
Here $\me\vert_{\GR{C}{n-2}}=4\vp_1+4\odin$. If $n\ge 4$, then $(4\vp_1)\md \GR{C}{n-2}\simeq
\mathbb A^6$, and it is isomorphic to $\vp'\vp''+2\odin$ as $\GR{A}{1}\times\GR{A}{1}$-module. Hence
$\me\md \GR{C}{n-2}\simeq \vp'\vp''+2\odin+\vp'\vp''$. It is easily seen that
$(2\vp'\vp'')\md \GR{A}{1}\times\GR{A}{1}\simeq \mathbb A^3$ (it is also {\rus N0}\,1 with $n=2$).
Therefore, $\me\md H\simeq \mathbb A^5$. By~\cite{g78b}, $(\GR{C}{n-2}, 4\vp_1)$ is equidimensional
if and only if $n\ge 5$. Then both quotient morphisms
\[
  \me \to \me\md \GR{C}{n-2} \to (\me\md \GR{C}{n-2})\md (\GR{A}{1}\times\GR{A}{1})=\me\md H
\]
are equidimensional. This already shows that $\me\md H$ is an affine space for $n\ge 4$ and, moreover, 
$(H:\me)$ is equidimensional, if $n\ge 5$. Some other {\sl ad hoc\/} methods allow us to handle the case 
with $n=3$ and prove the equidimensionality for $n=4$.

The other cases, where $H$ is semisimple, are $Sp_4\supset SL_2$ ({\rus N0}\,14) and
$\GR{B}{5}\supset \GR{B}{3}\times\GR{A}{1}$ ({\rus N0}\,12). 

-- \ In  {\rus N0}\,14, the $SL_2$-module 
$\me$ equals $\vp^6$ (binary forms of degree~6), and it is a classical fact from XIX century that 
$\vp^6\md SL_2$ is a hypersurface, cf.~also~\cite[Theorem\,4]{po83}.

-- \ In {\rus N0}\,12, $\me=\vp_3\vp'^2+\vp_1$ as $\GR{B}{3}\times\GR{A}{1}$-module. Then
$\me=3\vp_3+\vp_1$ as $\GR{B}{3}$-module and  $\me\md \GR{B}{3}\simeq \mathbb A^{10}$.
Using explicit multi-degrees of basic $\GR{B}{3}$-invariants, see {\rus N0}\,6 in~\cite[Table\,3]{ag79}, one 
sees that $\me\md \GR{B}{3}\simeq \vp'^2+\vp'^4+2\odin$ as $\GR{A}{1}$-module, and therefore
$(\me\md \GR{B}{3})\md \GR{A}{1}=\me\md H$ is a hypersurface.

Consider an item, where $H$ is not semisimple.
For {\rus N0}\,11, we have $(G,H)=(\GR{B}{4},\GR{G}{2}{\cdot}\mathsf T_1)$ and $\me=\vp_1\esi+\vp_1+\vp_1\esi^{-1}$. Then $\me=3\vp_1$ as $\GR{G}{2}$-module and $(3\vp_1)\md \GR{G}{2}\simeq\mathbb A^7$. Using explicit multi-degrees of basic $\GR{G}{2}$-invariants, cf. 
{\rus N0}\,1 in~\cite[Table\,4]{ag79}, we obtain that the $\mathsf T_1$-weights 
on $\mathbb A^7$ are $\esi^2,\esi,\esi^{-1},\esi^{-2}, 1,1,1$. Hence
$\mathbb A^7\md \mathsf T_1=\me\md H$ is a hypersurface.
\end{proof}

\begin{rmk}         \label{rem:ispravl}
I would like to fix some misprints and omissions in~\cite[Table~1]{p92}.

\textbullet \ In {\rus N0}\,1, the group $H$ has to be $SL_n\times SL_n$;

\textbullet \ the summand $\odin$ has to be added to $\me$ in {\rus N0}\,3,\,6. One also has $r=4$ in 
{\rus N0}\,3. 

\textbullet \ for {\rus N0}\,12, the right formula for $\me$ is given above.
\end{rmk}
A similar approach works for affine homogeneous spaces $\tG/\tH$ with $\tc=1$.

\begin{thm}     \label{thm:confirm-tc1}
If\/ $G$ is simple and $\tc=c_{\tG}(\tG/\tH)=1$, then $\g\md H$ is a hypersurface.
\end{thm}
\begin{proof}  We check the assertion for all items in Table~\ref{tc=1}. The $H$-modules $\g$ are 
given below. The underlined summands give rise to the adjoint representation of $H$.

1. \ $\slno=\un{\vp_1\vp_{n-1}}+\vp_1+\vp_{n-1}+\odin$ \ as $SL_n$-module.

2. \ $\mathfrak{sp}_6=\un{\vp_1^2}+\vp_1\vp'+\un{\vp'^2}$ \ as $Sp_4\times SL_2$-module.

3. \ $\mathfrak{so}_7=\vp_1+\un{\vp_2}$ \ as $\GR{G}{2}$-module.

4. \ $\GR{G}{2}=\un{\vp_1\vp_2}+\vp_1+\vp_2$ \ as $SL_3$-module.

5. \ $\mathfrak{sl}_3=\un{\vp^2}+\vp^4$ \ as $SL_2$-module \quad (we use the isomorphism 
$\tri\simeq\gt{so}_3$).

6. \ $\mathfrak{sl}_3=(\esi+\mu+\esi\mu)+(\esi+\mu+\esi\mu)^*+\un{2\odin}$ as $\mathsf T_2$-module.

7($s$). $\mathfrak{sp}_4=\vp^2\esi^2+ \un{\vp^2}+\vp^2\esi^{-2}+\un{\odin}$ \ as 
$SL_2{\cdot} \mathsf T_1$-module.

7($l$). $\mathfrak{sp}_4=\un{\vp^2}+\vp\esi+\vp\esi^{-1}+\esi^2+\esi^{-2}+\un{\odin}$ \ as $SL_2{\cdot} \mathsf T_1$-module.

8. \ $\mathfrak{sl}_4=\vp\vp'\esi+\vp\vp'\esi^{-1}+\un{\vp^2}+\un{\vp'^2}+\un{\odin}$ \ as 
$(SL_2\times SL_2){\cdot} \mathsf T_1$-module.

\noindent
$\blacktriangleright$ \ Items 1,\,3--5 are representations admitting a {\it finite coregular extension\/} in the sense of 
Shmel'kin~\cite{shm}, and he proves that here $\g\md H$ is an (explicitly described) hypersurface.

\noindent
$\blacktriangleright$ \ Items 7($l,s$)  can be handled in a similar way, and we provide details for one of them.
\\ \indent
{\bf --} \ In the $s$-case, we have
$\gt{sp}_4=3\vp^2+\odin$ as $SL_2$-module, and $(3\vp^2)\md SL_2$ is a hypersurface,
see~\cite[Theorem\,4]{po83}. We skip below the trivial $H$-module $\odin$.
It is not hard to write explicitly down the basic invariants 
for $(SL_2 : 3\vp^2)$. 
Let $F$ denote the basic invariant of degree 2 for the adjoint representation
$(SL_2:\vp^2)$, i.e., $F(v)=(v,v)$ for $v\in\vp^2$.
If $(v_1,v_2,v_3)\in 3\vp^2$, then the basic $SL_2$-invariants are:
$F_{ij}$, $1\le i\le j\le 3$, and $\tilde F$, where $F_{ij}(v_1,v_2,v_3)=(v_i,v_j)$ and 
$\tilde F=\det [v_1,v_2,v_3]$. The basic relation is
\beq   \label{eq:rel-sl(2)-inv}
          \det \bigl( (F_{ij})_{i,j=1}^3 \bigr)={\tilde F}^2 , 
\eeq
where $\boldsymbol{F}=(F_{ij})_{i,j=1}^3$ is the symmetric 3 by 3 matrix.
If  $t{\cdot}(v_1,v_2,v_3)=(t^{-2}v_1,v_2,t^{2}v_3)$ for $t\in \mathsf T_1$, then $F_{13}, F_{22}, \tilde F$ are already 
$\mathsf T_1$-invariants, but $t{\cdot}F_{11}=t^4 F_{11}$, $t{\cdot}F_{12}=t^2 F_{12}$, 
$t{\cdot}F_{23}=t^{-2} F_{23}$, and $t{\cdot}F_{33}=t^{-4} F_{33}$. Therefore, the 
other $SL_2{\cdot}\mathsf T_1$-invariants are:
\[
  x_1=F_{11}F_{33}, \ y_1=F_{12}F_{23}, \ z_1=F_{11}F_{23}^2, \ z_2=F_{33}F_{12}^2 .
\]
Thus, we get seven generators and yet another relation 
\beq            \label{eq:rel-sl(2)-inv2}
               x_1y_1^2=z_1z_2 . 
\eeq
Expressing the $\det\bigl(\boldsymbol{F}\bigr)$ via these $SL_2{\cdot}\mathsf T_1$-invariants, we rewrite 
\eqref{eq:rel-sl(2)-inv} as
\[
   x_1F_{22}+2y_1F_{13}+F_{13}^2F_{22}+z_1+z_2=\tilde F^2 .
\]
Therefore, either $z_1$ or $z_2$ can be excluded from the minimal generating system of $\bbk[\me]^H$.
Afterwards, \eqref{eq:rel-sl(2)-inv2} yields the relation for the remaining six invariants.

{\bf --} \ In the $l$-case, $\gt{sp}_4=\vp^2+2\vp+3\odin$ as $SL_2$-module, and $(\vp^2+2\vp)\md SL_2$ is a hypersurface, too. The rest is similar to the $s$-case.

\noindent
$\blacktriangleright$ \ {\rus N0}\,6 is easy and left to the reader.

\noindent
$\blacktriangleright$ \ {\rus N0}\,8 is a slice representation for {\rus N0}\,2. Therefore, using the 
monotonicity results for homological dimension of algebras of invariants~\cite[Theorem\,2]{po83}, it 
suffices to handle {\rus N0}\,2.  

\noindent
$\blacktriangleright$ \ {\rus N0}\,2 is the most difficult case, and we only give some hints.
Here $\gt{sp}_6={\vp_1^2}+2\vp_1+3\odin$ as $Sp_4$-module.
The representation  
$(Sp_4 : {\vp_1^2}+2\vp_1)$ is a slice for $(SL_4: {\tvp_1^2}+\tvp_2+2\tvp_1^*=\tilde V)$ 
(use $v\in \tvp_2$ such that $(SL_4)^v=Sp_4$)
and $\hd(\tilde V\md SL_4)=2$~\cite[Table\,9]{shm01}. Therefore, $({\vp_1^2}+2\vp_1)\md Sp_4$ is a complete intersection and  $\hd ({\vp_1^2}+2\vp_1)\md Sp_4\le 2$~\cite{po83}. The subsequent argument is similar in spirit with that in case 7($s$), but much more elaborated. We also need the fact that, for the
truncated $Sp_4\times SL_2$-module ${\vp_1^2}+\vp_1\vp'=\VV\subset\gt{sp}_6$, the quotient $\VV\md (Sp_4\times SL_2)$ is an affine space of dimension~5~\cite{w93}.
\end{proof}

\section{Coisotropy representations and related Poisson structures}
\label{sect:poisson}

\noindent
In this section, $G/H$ is affine, and we think of $\me$ as a subspace of $\g^*$. The cotangent bundle 
$T^*_{G/H}=G\times_H\me$ is a symplectic $G$-variety, hence the algebra $\bbk[T^*_{G/H}]$ is 
equipped with the associated Poisson bracket $\{\ ,\, \}$. This bracket restricts to the algebra of 
$G$-invariants $\bbk[T^*_{G/H}]^G\simeq \bbk[\me]^H$, which makes $\me\md H$ a Poisson variety.  
Recall that $\dim\me\md H=2c+r$.

The Poisson bracket $(\bbk[\me]^H, \{\ ,\, \})$ has the following explicit description, see 
e.g.~\cite[Ch.\,II, \S\,1.8]{v01}.
The algebra of regular functions $\bbk[\me]$ is also the symmetric algebra of $\g/\h\simeq \me^*$, hence
$\bbk[\me]^H=\mathcal S(\g/\h)^H$. Let $f_1,f_2\in \mathcal S(\g/\h)^H$ and 
$\ap\in (\g/\h)^*=\me\subset\g^*$. Then
\beq          \label{eq:vinb01}
      \{f_1,f_2\}(\ap)=\langle \ap, [\textsl{d}_\ap f_1, \textsl{d}_\ap f_2]\rangle .
\eeq
The commutator on the right-hand side of this formula is understood as the commutator in $\g$ of any 
representatives of the cosets $\textsl d_\ap f_1, \textsl d_\ap f_2\in \g/\h$, because the result does not 
depend on the choice of these representatives in $\g$ (if $f_1$ and $f_2$ are $H$-invariants!).
Note that if $\h=\g^\sigma$ for an involution $\sigma$, then the right-hand side in~\eqref{eq:vinb01} 
is identically zero. That is, for the symmetric variety $G/H$,
the Poisson bracket vanishes on $\bbk[T^*_{G/H}]^G$. Let $\rk \{\ ,\, \}$ denote the {\it rank of the 
Poisson bracket}, i.e., the maximal dimension of symplectic leaves in $\me\md H$. It follows from~\cite[Ch.\,II, \S\,3, Theorem\,2]{v01} that in our case $\rk \{\ ,\, \}=2c$.

Let $(\eus P, \{\ ,\, \}_{\eus P})$ be an affine Poisson variety.
A subalgebra $\gA$ of $\bbk[\eus P]$ is said to be {\it Poisson-commutative}, if $\{\gA,\gA\}=0$. 
As is well-known, if $\gA$ is Poisson-commutative, then
\[ 
    \trdeg\gA\le \dim \eus P-\frac{1}{2}\rk\! \{\ ,\, \}_{\eus P} .
\]
Therefore, we arrive at the following conclusion.  
\begin{lm}    \label{lm:trdeg-A}
If\/ $\gA$ is a Poisson-commutative subalgebra of\/ $\bbk[\me]^H$, then $\trdeg\gA\le c+r$.
\end{lm}

\begin{conj}   \label{conj:upper}
For any affine homogeneous space $G/H$,  there is a Poisson-commutative subalgebra
$\gA\subset \bbk[\me]^H$ such that $\trdeg\gA=c+r$.
\end{conj}

Let $\gZ$ denote the Poisson centre of $(\bbk[\me]^H, \{\ ,\, \})$. By~\cite[Section\,7]{kn90}, $\eus Z$ is 
a polynomial ring and $\trdeg\gZ=r$. Some stronger results can also be found in \cite[Section\,9]{kn94}.

\begin{ex}
For $c=0$, one has $\gZ=\bbk[\me]^H$, and there is nothing to prove.
For $c=1$, $\trdeg\bbk[\me]^H=\trdeg\gZ+2$ and
one can take any $f\in \bbk[\me]^H$ that is not algebraic over $\gZ$.
Then the subalgebra generated by $\gZ$ and $f$ is Poisson-commutative and its transcendence degree
equals $r+1$, as required. Thus, Conjecture~\ref{conj:upper} is true, if $c\le 1$.
\end{ex}

By~\cite[Theorem\,7.6]{kn90}, one has $\gZ=\bbk[M_\co]^G$ and the morphism 
$T^*_{G/H}\to \spe\gZ$ is given by the map $\psi$ in \eqref{diagr-knop}. Therefore, using 
commutative diagrams~\eqref{diagr-knop} and \eqref{diagr-knop2}, we get the chain morphisms
\[
   \me\stackrel{\pi_\me}{\longrightarrow} \me\md H\longrightarrow 
   \spe\gZ=M_\co\md G\stackrel{\tau}{\longrightarrow} \ov{G{\cdot}\me}\md G ,
\]
where $\tau$ is finite, the morphisms $\boldsymbol{f}: \me \to \spe \gZ$ and
$\boldsymbol{\tilde f}: \me \to \ov{G{\cdot}\me}\md G=\tilde M_\co\md G$ are equidimensional, and
$\boldsymbol{\tilde f}^{-1}(\bar 0)=\boldsymbol{f}^{-1}(\bar 0)=\fngs\cap\me$.

\subsection{The case of $\tG/\tH$}
For the homogeneous spaces of the form $\tG/\tH=(G\times H)/\Delta_H$, one can say more. Recall that 
$\tr=\rk\g+\rk\h$, $\tc=\dim U-\dim B_H$,  and $\tme\simeq \g^*$. Here Lemma~\ref{lm:trdeg-A} says 
that if $\gA\subset \bbk[\g^*]^H=\gS(\g)^H$ is Poisson-commutative, then
\beq    \label{eq:upper}
       \trdeg \gA\le\tc+\tr=\frac{1}{2}(\dim\g-\dim\h+\rk\g+\rk\h). 
\eeq
In this setting, the existence of $\gA$ such that $\trdeg\gA=\tc+\tr$ has been proved for several classes 
of reductive subalgebras $\h$ :

\textbullet \ \ $\h=\g^\sigma$ is a symmetric subalgebra~\cite[Theorem\,2.7]{py21a};

\textbullet \ \ $\h=\g^\theta$, where $\vartheta$ is an automorphism of $\g$ of finite order 
$\ge 3$~\cite[Theorem\,3.10]{py21b}. Here one also needs the condition that a certain contraction 
of  $\g$ associated with $\vartheta$, denoted $\g_{(0)}$, has the same index as $\g$. It should be noted,
however, that this condition has been verified in many cases, and it is likely that this condition always 
holds.

\textbullet \ \ $\h$ is the centraliser of a semisimple element of $\g$~\cite[Lemma\,2.1]{my19}, i.e., $\h$ is a Levi subalgebra of $\g$. 

Note also that if $\h\subset\g$ has a non-trivial centre, then, for any $\rr$ such that 
$[\h,\h]\subset \rr\subset \h$, we have $\dim\h-\rk\h=\dim\rr-\rk\rr$. Therefore, if a Poisson-commutative 
subalgebra $\gA\subset \gS(\g)^\h$ has the maximal transcendence degree, then it follows 
from~\eqref{eq:upper} that $\gA$ has the maximal transcendence degree as subalgebra of the larger 
Poisson algebra $\gS(\g)^\rr$.

{\bf Remark.} An advantage of this case is that $\bbk[\tme]=\gS(\g)$ is a Poisson algebra.
Therefore, one can construct `large' Poisson-commutative subalgebras in $\gS(\g)^H$ using compatible Poisson brackets and Mishchenko-Fomenko 
subalgebras of $\gS(\g)$, see~\cite{my19,py21a,py21b}. But for an arbitrary
affine $G/H$, the algebra $\bbk[\me]$ does not possess a natural Poisson structure. 

\subsection{A more general setting}
\label{subs:my19}
Let $R\subset Q$ be arbitrary connected affine algebraic groups. Then $Q{\times} R/\Delta_R\simeq Q$
is an affine homogeneous space of  $Q{\times} R$ and the coisotropy representation of $R\simeq \Delta_R$ is isomorphic to $(R:\q^*)$. Here we are led to consider Poisson-commutative subalgebras
of the Poisson algebra $\gS(\q)^R=\gS(\q)^\rr$. 

Our luck is that this problem (without connection to coisotropy representations) has been considered in~\cite{my19}, where an upper bound on $\trdeg\gA$ similar to \eqref{eq:upper} is given. The only 
difference is that the {\it rank\/} of a Lie algebra has to be replaced with the {\it index\/}. (Recall that 
$\rk\q=\ind\q$ whenever $\q$ is reductive.)
That is, if $\gA\subset \gS(\q)^\rr$ and $\{\gA,\gA\}=0$, then
\beq    \label{eq:upper2}
       \trdeg \gA\le \frac{1}{2}(\dim\q-\dim\rr+\ind\q+\ind\rr) , 
\eeq
see~\cite[Prop.\,1.1]{my19}. It is also shown in~\cite{my19} that if $\rr=\q^\xi$ is the stabiliser of 
$\xi\in\q^*$ under the coadjoint representation of $\q$ and $\ind\q^\xi=\ind\q$, then this bound is achieved in many cases. In particular, if $\q$ is reductive, then this 
happens for any $\xi$. 

However, results of \cite{kn90} do not apply in this setting, unless both groups $Q$ and $R$ are reductive.

\end{document}